\documentclass{birkjour}
\usepackage{graphicx}
 \newtheorem{thm}{Theorem}[section]
 \newtheorem{cor}[thm]{Corollary}
 \newtheorem{lem}[thm]{Lemma}

 \newtheorem{defn}[thm]{Definition}
 \newtheorem{rem}[thm]{Remark}
 \newtheorem{ex}{Example}
\vspace{2cm}

\begin{document}

%
%
%
%
%
%
%
%
%

\title[Analytical Study of a Class of  Rational Difference
Equations]
 {Analytical Study of a Class of  Rational Difference
Equations \vspace{2cm}}

\keywords{Difference equations, Good set, Convergence, Unbounded
solutions, Periodic solutions}

\date{}
\author{Fethi Kadhi}

\address{%
University of Manouba\\
\'Ecole Nationale des Sciences de l'Informatique
Department of Mathematics\\
Saudi Arabia}

\email{fethi.kadhi@ensi-uma.tn}

\author{Malek Ghazel}
\address{University of Hail\\
Faculty of Science\\
Department of Mathematics\\
Saudi Arabia}
 \email{malek\_ghazel@yahoo.fr, m.ghazel.uoh.edu.sa}
\subjclass{34K05, 34K13, 34K20, 39A10}

\keywords{Difference equations, Equilibrium point, Global
stability, Hyperbolic, Repeller, Oscillatory, Periodic solutions,
Numerical simulation}

\date{}
\begin{abstract}
We obtain the solution of the fourth order difference equation $$
x_{n+1}=\frac{ \alpha x_{n-3}}{A+B x_{n-1}x_{n-3}}$$ with the
initial conditions; $x_{-3}=d,$ $x_{-2}=c,$ $x_{-1}=b,$ and
$x_{0}=a$ are arbitrary nonzero real numbers, $\alpha$, $A$ and
$B$ are arbitrary constants. The result is used to study the
convergence of solutions, the existence of unbounded solutions and
the convergence to  periodic solutions. We illustrate the results
by several numerical examples.
\end{abstract}
\vspace{2cm}
\maketitle
\section{Introduction} Difference equations arise in the study of
the evolution of dynamical systems. Their applications to various
fields are rapidly increasing. For example, they are frequently
used in numerical analysis to describe the different schemes of
approximation, \cite{k1}. Moreover, difference equations are of
interest in themselves, especially in view of remarkable analogy
with the theory of differential equations. According to this
analogy, there are two manners to treat a difference equation: The
first is to study the qualitative behavior of solutions, for
examples:

\noindent Camouzis and Ladas \cite{cala} considered the dynamics
of the third-order rational difference equation
$$x_{n+1}=\frac{\alpha+\beta x_n+\gamma x_{n-1}+\delta
x_{n-2}}{A+Bx_n+Cx_{n-1}+Dx_{n-2}}$$ with nonnegative parameters
$\alpha, \beta, \gamma,\delta, A,B,C,D$ and with arbitrary
nonnegative initial conditions $x_{-2}, x_{-1}, x_0$.\\
\noindent Elabbasy et all \cite{EL AB Elmetw & EM EL 1} studied
the qualitative behavior of the difference equations $$x_{n+1}=a
x_{n}-\frac{bx_{n}}{cx_{n}-dx_{n-1}}\:\mbox{ and }\:
x_{n+1}=\frac{\alpha x_{n-k}}{\beta+\gamma
\Pi_{i=0}^{k}x_{n-i}}.$$ \noindent The second manner to treat a
difference equation is to find the explicit formula of solutions.
Contrary to the linear case, there is no general method to find
such explicit solution. However,  Cinar \cite{C C 2} obtained the
solution of the difference equations
 $$x_{n+1}=\frac{x_{n-1}}{-1+ x_{n}x_{n-1}}\:\mbox{  and   } \: x_{n+1}=\frac{ax_{n-1}}{-1+ bx_{n}x_{n-1}}.$$
\noindent Elsayed \cite{EM EL 12} solved the equations
$$x_{n+1}=\frac{x_{n-3}}{\pm 1 \pm x_{n-1}x_{n-3}}$$

 In this paper, we solve the following class of difference
 equations:
\begin{equation}\label{e_1}
x_{n+1}=\frac{\alpha x_{n-3}}{A+B x_{n-1}x_{n-3}}
\end{equation}
with the initial data : $x_{-3}=d,$ $x_{-2}=c,$ $x_{-1}=b,$ and
$x_{0}=a$ are nonzero real numbers. We use the obtained result to
determine the forbidden set of initial conditions and to discuss
the convergence of solutions. Depending on coefficients $\alpha$
and $A$, the existence of unbounded solutions and the convergence
to periodic solutions are studied. Our results are confirmed by
numerical examples.
\section{ Solution of equation \eqref{e_1}}
The following theorem gives an analytical expression of the
solution of $\eqref {e_1}$.
\begin{thm} \label{Th_1}Let $(x_n)_{n=-3}^\infty$ be the solution of $\eqref {e_1}$,
then, for all $n\geq 2$
\begin{align}\label{e_2}x_{4n-3}&=\frac{d\alpha^n\prod_{p=0}^{n-2}\Big(A^{2p+2}+B bd\sum_{i=0}^{2p+1}A^i\alpha^{2p+1-i}\Big)}
{\prod_{p=0}^{n-1}\Big(A^{2p+1}+B
bd\sum_{i=0}^{2p}A^i\alpha^{2p-i}\Big)},\\
x_{4n-2}&=\frac{c\alpha^n\prod_{p=0}^{n-2}\Big(A^{2p+2}+B
ac\sum_{i=0}^{2p+1}A^i\alpha^{2p+1-i}\Big)}
{\prod_{p=0}^{n-1}\Big(A^{2p+1}+B
ac\sum_{i=0}^{2p}A^i\alpha^{2p-i}\Big)},\\
 \label{e_3}
x_{4n-1}&=\frac{b\alpha^n\prod_{p=0}^{n-1}\Big(A^{2p+1}+
Bbd\sum_{i=0}^{2p}A^i\alpha^{2p-i}\Big)}{\prod_{p=0}^{n-1}\Big(A^{2p+2}+B
bd\sum_{i=0}^{2p+1}A^i\alpha^{2p+1-i}\Big)},\\
\label{e5} x_{4n}&=\frac{a\alpha^n\prod_{p=0}^{n-1}\Big(A^{2p+1}+B
ac\sum_{i=0}^{2p}A^i\alpha^{2p-i}\Big)}
{\prod_{p=0}^{n-1}\Big(A^{2p+2}+B
ac\sum_{i=0}^{2p+1}A^i\alpha^{2p+1-i}\Big)}.\end{align}
\end{thm}
\begin{proof}
By induction, we will prove the result for $x_{4n-3}$. For $n=2$,
it is easy to check that
$x_{5}=\frac{d\alpha^2\prod_{p=0}^{0}\Big(A^{2p+2}+B
bd\sum_{i=0}^{2p+1}A^i\alpha^{2p+1-i}\Big)}{\prod_{p=0}^{1}\Big(A^{2p+1}+B
bd\sum_{i=0}^{2p}A^i\alpha^{2p-i}\Big)}$. Let $n\geq 2$. Suppose
that the result holds at the step $n$ and let us prove the result
for the step $n+1$,
\begin{align*}
&x_{4(n+1)-3}=x_{4n+1}=\frac{\alpha x_{4n-3}}{A+B
x_{4n-1}x_{4n-3}}\\
&=\frac{ d\alpha^{n+1}\prod_{p=0}^{n-1}\Big(A^{2p+2}+B
bd\sum_{i=0}^{2p+1}A^i\alpha^{2p+1-i}\Big)}
{\prod_{p=0}^{n-1}\Big(A^{2p+1}+B
bd\sum_{i=0}^{2p}A^i\alpha^{2p-i}\Big)\Big[A\Big(A^{2n}+B
bd\sum_{i=0}^{2n-1}A^i\alpha^{2n-1-i}\Big)+B
bd\alpha^{2n}\Big]}\\
&=\frac{ d\alpha^{n+1}\prod_{p=0}^{n-1}\Big(A^{2p+2}+B
bd\sum_{i=0}^{2p+1}A^i\alpha^{2p+1-i}\Big)}
{\prod_{p=0}^{n-1}\Big(A^{2p+1}+B
bd\sum_{i=0}^{2p}A^i\alpha^{2p-i}\Big)\Big(A^{2n+1}+B
bd\Big(\sum_{i=0}^{2n-1}A^{i+1}\alpha^{2n-1-i}+
\alpha^{2n}\Big)\Big)}\\
&=\frac{ d\alpha^{n+1}\prod_{p=0}^{n-1}\Big(A^{2p+2}+B
bd\sum_{i=0}^{2p+1}A^i\alpha^{2p+1-i}\Big)}
{\prod_{p=0}^{n-1}\Big(A^{2p+1}+B
bd\sum_{i=0}^{2p}A^i\alpha^{2p-i}\Big)\Big(A^{2n+1}+B
bd\Big(\sum_{i=1}^{2n}A^{i}\alpha^{2n-i}+ \alpha^{2n}\Big)\Big)}\\
&= \frac{ d\alpha^{n+1}\prod_{p=0}^{n-1}\Big(A^{2p+2}+B
bd\sum_{i=0}^{2p+1}A^i\alpha^{2p+1-i}\Big)}
{\prod_{p=0}^{n}\Big(A^{2p+1}+B
bd\sum_{i=0}^{2p}A^i\alpha^{2p-i}\Big)}
\end{align*}
Similarly, we prove the other formulas.
\end{proof}
For every $n\in \mathbb{N}$, $\alpha$, $A$, $B$, $a$, and $b$,
denote
$$P_{p,A,\alpha}^{B,b,d}=\Big(A^{p}(A-\alpha+B bd)-B bd\alpha^{p}\Big).$$
The following corollary gives a simplified analytic expression of
the solution when $A\neq\alpha$.
\begin{cor}\label{c_1}
 If $A\neq\alpha$, then the subsequences of the solution of \eqref{e_1} can be written as:
$$x_{4n-3}=\frac{d\alpha^n (A-\alpha)\prod_{p=0}^{n-2}P_{2p+2,A,\alpha}^{B,b,d}}{\prod_{p=0}^{n-1}P_{2p+1,A,\alpha}^{B,b,d}},\qquad
x_{4n-2}=\frac{c\alpha^n
(A-\alpha)\prod_{p=0}^{n-2}P_{2p+2,A,\alpha}^{B,a,c}}{\prod_{p=0}^{n-1}P_{2p+1,A,\alpha}^{B,a,c}},$$
$$x_{4n-1}=\frac{b\alpha^n \prod_{p=0}^{n-1}P_{2p+1,A,\alpha}^{B,b,d}}{\prod_{p=0}^{n-1}P_{2p+2,A,\alpha}^{B,b,d}},\qquad
x_{4n}=\frac{a\alpha^n
\prod_{p=0}^{n-1}P_{2p+1,A,\alpha}^{B,a,c}}{\prod_{p=0}^{n-1}P_{2p+2,A,\alpha}^{B,a,c}}.$$
\end{cor}
\begin{proof} It suffices to use the binomial identity $x^p-y^p=(x-y)\sum_{i=0}^{p-1}x^iy^{p-1-i}$
 in the analytical expression of the subsequences \eqref{e_2}, $\cdots$, \eqref{e5}.
\end{proof}
If $A=\alpha$, then every solution of $\eqref{e_1}$ can be
expressed by using the function Gamma.
\begin{cor}\label{c_2}
If $A=\alpha$, then
\begin{align*}
x_{4n-3}&=\frac{A 2^{2n-2}\Gamma^2(\frac{A}{2B
bd}+n)\Gamma(\frac{A}{B bd}+1)}{B b\Gamma^2(\frac{A}{2B
bd}+1)\Gamma(\frac{A}{B bd}+2n)},\\
 x_{4n-2}&=\frac{A 2^{2n-2}\Gamma^2(\frac{A}{2B
ac}+n)\Gamma(\frac{A}{B ac})}{B a\Gamma^2(\frac{A}{2B
ac}+1)\Gamma(\frac{A}{B ac}+2n)},\\
x_{4n-1}&=\frac{b\Gamma(\frac{A}{B bd}+2n+1)\Gamma^2(\frac{A}{2B
bd}+1)}{2^{2n}\Gamma(\frac{A}{B bd}+1)\Gamma^2(\frac{A}{2B
bd}+n+1)},\\
 x_{4n}&=\frac{a\Gamma(\frac{A}{B ac}+2n+1)\Gamma^2(\frac{A}{2B
ac}+1)}{2^{2n}\Gamma(\frac{A}{B ac}+1)\Gamma^2(\frac{A}{2B
ac}+n+1)}.\end{align*}
\end{cor}
\begin{proof}
By \eqref{e_2} we have:

$x_{4n-3}=\frac{ dA^n\prod_{p=0}^{n-2}\Big(A^{2p+2}+B
bd\sum_{i=0}^{2p+1}A^{2p+1}\Big)}{\prod_{p=0}^{n-1}\Big(A^{2p+1}+B
bd\sum_{i=0}^{2p}A^{2p}\Big)}=\frac{
dA^n\prod_{p=0}^{n-2}A^{2p+1}\Big(A+(2p+2)B
bd\Big)}{\prod_{p=0}^{n-1}A^{2p}\Big(A+(2p+1)B bd\Big)}$

$=\frac{
dA^n\prod_{p=0}^{n-2}A^{2p}\prod_{p=0}^{n-2}A\prod_{p=0}^{n-2}\Big(A+(2p+2)B
bd\Big)}{\prod_{p=0}^{n-1}A^{2p}\prod_{p=0}^{n-1}\Big(A+(2p+1)B
bd\Big)}=\frac{ dA\prod_{p=0}^{n-2}\Big(A+(2p+2)B
bd\Big)}{\prod_{p=0}^{n-1}\Big(A+(2p+1)B bd\Big)}$

$=\frac{
dA\prod_{p=0}^{n-2}Bbd\Big(\frac{A}{Bbd}+2p+2\Big)}{\prod_{p=0}^{n-1}Bbd\Big(\frac{A}{Bbd}+2p+1\Big)}=\frac{
dA\prod_{p=1}^{n-1}\Big(\frac{A}{Bbd}+2p\Big)}{
Bbd\prod_{p=0}^{n-1}\Big(\frac{A}{Bbd}+2p+1\Big)}$

$=\frac{
A\Big[\prod_{p=1}^{n-1}\Big(\frac{A}{Bbd}+2p\Big)\Big]^2}{
Bb\prod_{p=0}^{n-1}\Big(\frac{A}{Bbd}+2p+1\Big)\Big(\prod_{p=1}^{n-1}\Big(\frac{A}{Bbd}+2p\Big)}=\frac{
A\Big[\prod_{p=1}^{n-1}2\Big(\frac{A}{2Bbd}+p\Big)\Big]^2}{
Bb\prod_{p=1}^{2n-1}\Big(\frac{A}{Bbd}+p\Big)}$

$=\frac{
A2^{2n-2}\Big[\prod_{p=1}^{n-1}\Big(\frac{A}{2Bbd}+p\Big)\Big]^2\Gamma(\frac{A}{Bbd}+1)}{
Bb\,\Gamma\Big(\frac{A}{Bbd}+2n\Big)}=\frac{
A2^{2n-2}\Gamma^2\Big(\frac{A}{2Bbd}+n\Big)\Gamma(\frac{A}{Bbd}+1)}{
Bb\,\Gamma\Big(\frac{A}{Bbd}+2n\Big)\Gamma^2\Big(\frac{A}{2Bbd}+1\Big)}.$

Similarly, one can prove the other relations. This ended the
proof.
\end{proof}
\begin{rem}\label{Rem_1}
\begin{enumerate}\noindent
\item The cases $\alpha =0$ or $B=0$ are trivial. Then, $\alpha$
and $\beta$ are assumed to be nonzero real numbers.
 \item A common hypothesis in the study of rational
difference equations is the choice of positive coefficients and
initial conditions. Therefore, all the solutions will be
automatically well defined. This is the framework, for example, in
\cite{cala}. It is a problem of great difficulty to determine the
good set of initial conditions for which a solution of a rational
difference equation is well defined for all $n\geq 0$. The
obtention of the explicit solution helps to extend the choice of
coefficients and initial conditions. By corollary, \eqref{c_2}, If
$abcd\neq 0$,$\frac{A}{Bbd},\: \frac{A}{Bac}\notin
\{1\}\cup\{2n,n\in\mathbb{Z}\}$, then all the solutions of
\eqref{e_1} are well defined.
 \item Our results cover the four equations considered in
Elsayed \cite{EM EL 12}. For example, the equation
$$x_{n+1}=\frac{x_{n-3}}{1+x_{n-1}x_{n-3}}$$ corresponds to the
case $\alpha= A=B=1$. The obtained expressions, in this case, are:
$$x_{4n-3}=\frac{ d\prod_{i=0}^{n-1}\Big(1+2ibd\Big)}{ \prod_{i=0}^{n-1}\Big(1+(2i+1)bd\Big)},\qquad x_{4n-1}=\frac{ b\prod_{i=0}^{n-1}\Big(1+(2i+1)bd\Big)}{ \prod_{i=0}^{n-1}\Big(1+(2i+2)bd\Big)},$$
 $$x_{4n-2}=\frac{ c\prod_{i=0}^{n-1}\Big(1+2iac\Big)}{ \prod_{i=0}^{n-1}\Big(1+(2i+1)ac\Big)},\qquad x_{4n}=\frac{ a\prod_{i=0}^{n-1}\Big(1+(2i+1)ac\Big)}{ \prod_{i=0}^{n-1}\Big(1+(2i+2)ac\Big)},$$
where $\prod_{i=0}^{-1}A_i=1$.

According to our results, these expressions can be rewritten using
the Gamma function as:
$$x_{4n-3}=\frac{2^{2n-2}\Gamma^2(\frac{1}{2 bd}+n)\Gamma(\frac{1}{bd})}{b\Gamma^2(\frac{1}{2bd}+1)\Gamma(\frac{1}{bd}+2n)},\qquad x_{4n-2}=\frac{2^{2n-2}\Gamma^2(\frac{1}{2ac}+n)\Gamma(\frac{1}{ac})}{a\Gamma^2(\frac{1}{2ac}+1)\Gamma(\frac{1}{ac}+2n)},$$
$$x_{4n-1}=\frac{b\Gamma(\frac{1}{bd}+2n+1)\Gamma^2(\frac{1}{2bd}+1)}{2^{2n}\Gamma(\frac{1}{bd}+1)\Gamma^2(\frac{1}{2bd}+n+1)},\qquad
x_{4n}=\frac{a\Gamma(\frac{1}{ac}+2n+1)\Gamma^2(\frac{1}{2ac}+1)}{2^{2n}\Gamma(\frac{1}{ac}+1)\Gamma^2(\frac{1}{2ac}+n+1)}.$$
\end{enumerate}
\end{rem}
\section{Convergence and existence of unbounded solutions}
\subsection{The case $|\frac{A}{\alpha}|>1$}

\begin{thm}\label{Th_3}
Assume that $|\frac{A}{\alpha}|>1$, then:
\begin{enumerate}
\item If $A-\alpha+B bd\neq 0$ and $A-\alpha+B ac\neq 0$, then
every solution of the equation \eqref{e_1} converges toward zero.

\item If  $A-\alpha+B bd= 0$ and $A-\alpha+B ac= 0$, then every
solution of the equation \eqref{e_1} converges to a number $l$ if
and only if $a=b=c=d=l$
\end{enumerate}
\end{thm}
\begin{proof}
\begin{enumerate}
\item By corollary \eqref{c_1}
\begin{align*}
x_{4n-3}&=\frac{d\alpha^n(A-\alpha)\prod_{p=0}^{n-2}\Big(A^{2p+2}(A-\alpha+B
bd)-B
bd\alpha^{2p+2}\Big)}{\prod_{p=0}^{n-1}\Big(A^{2p+1}(A-\alpha+B
bd)-B bd\alpha^{2p+1}\Big)}\\
&=\frac{d\alpha^n(A-\alpha)\prod_{p=0}^{n-2}A^{2p+2}(A-\alpha+B
bd)\prod_{p=0}^{n-2}\Big(1-\frac{B bd}{A-\alpha+B
bd}(\frac{\alpha}{A})^{2p+2}\Big)}{\prod_{p=0}^{n-1}A^{2p+1}(A-\alpha+B
bd)\prod_{p=0}^{n-1}\Big(1-\frac{B bd}{A-\alpha+B
bd}(\frac{\alpha}{A})^{2p+1}\Big)}\\
&=\frac{d\alpha^n(A-\alpha)A^{n-1}\prod_{p=0}^{n-2}\Big(1-\frac{B
bd}{A-\alpha+B bd}(\frac{\alpha}{A})^{2p+2}\Big)}{(A-\alpha+B
bd)A^{2n-1}\prod_{p=0}^{n-2}\Big(1-\frac{B bd}{A-\alpha+B
bd}(\frac{\alpha}{A})^{2p+1}\Big)}
\end{align*}
Denote by $\beta=\frac{B bd}{A-\alpha+B bd}$ and let $(U_p)_p$ be
the sequence defined by $U_p=\frac{1-\beta
(\frac{\alpha}{A})^{2p+2}}{1-\beta (\frac{\alpha}{A})^{2p+1}}$.
then we can write
$$x_{4n-3}=\frac{d(\frac{\alpha}{A})^n(A-\alpha)}{(A-\alpha+B bd)\Big(1-\beta(\frac{\alpha}{A})^{2n-1}\Big)}\prod_{p=0}^{n-2}U_p.$$
 For $p\in\mathbb{N}$ big enough, We have either:  $U_p>1$ or
 $0<U_p<1$.
By Taylor expansion, we obtain
\begin{align*}
U_p&=(1-\beta (\frac{\alpha}{A})^{2p+2})(1+\beta
(\frac{\alpha}{A})^{2p+1}+o(\frac{\alpha}{A})^{2p+1})\\
&=1+\beta
(1-\frac{\alpha}{A})(\frac{\alpha}{A})^{2p+1}+o(\frac{\alpha}{A})^{2p+1}
\end{align*}
then $U_p\sim 1+\beta(1-\frac{\alpha}{A})
(\frac{\alpha}{A})^{2p+1}$ which is the general term of a
convergent infinite product. We deduce that  $(x_{4n-3})_n$
converges toward zero. Similarly we do for the other subsequences.
\item If  $A-\alpha+B bd= 0$ and $A-\alpha+B ac= 0$, then the
subsequences $(x_{4n-3})_n$, $(x_{4n-1})_n$ are constant:
$x_{4n-3}\equiv d$ and $x_{4n-1}\equiv b$. Similarly, the
subsequences $(x_{4n-2})_n$, $(x_{4n})_n$ are constant:
$x_{4n-2}\equiv c$ and  $x_{4n}\equiv a$. Then every solution of
the equation \eqref{e_1} converges to a number $l$ if and only if
$a=b=c=d=l$.
\end{enumerate}
\end{proof}
\subsection{The case $|\frac{A}{\alpha}|=1$}
We distinguish two cases: $A=\alpha$ and $A=-\alpha$.
\begin{thm}\label{Th_4}
If $A=\alpha$, then every solution of equation \eqref{e_1}
converges toward zero.
\end{thm}
\begin{proof}
 Denote $e=\frac{A}{Bbd}$. From the proof of corollary \eqref{c_2}, we
 deduce:
 \begin{align*}
 x_{4n-3}&=\frac{ A\prod_{p=1}^{n-1}\Big(\frac{A}{Bbd}+2p\Big)}{
 Bb\prod_{p=0}^{n-1}\Big(\frac{A}{Bbd}+2p+1\Big)}\\
 &=\frac{ A}{Bb(e+1)}
 \prod_{p=1}^{n-1}\Big(\frac{ e+2p}{
 e+2p+1}\Big)\\
&=\frac{ A}{Bb(e+1)}\prod_{p=1}^{n-1} \Big(\frac{ \frac{e}{2p}+1}{
\frac{e+1}{2p}+1}\Big)
\end{align*}
Let $(W_p)_p$ be the sequence defined by
$W_p=\frac{\frac{e}{2p}+1}{\frac{e+1}{2p}+1}$, it is clear that:
(i) $\lim_{p\rightarrow\infty}W_p=1$,\\
(ii) For $p$ big enough, $0<W_p<1$.\\
\noindent By Taylor expansion,
$W_p=(1+\frac{e}{2p})(1-\frac{e+1}{2p}+o(\frac{1}{p}))=1-\frac{1}{2p}+o(\frac{1}{p})$,
which is a general term of a divergent infinite product, since for
$p$ big enough, $0<W_p<1$, then $\lim_{n\rightarrow
\infty}\Pi_{p=1}^{n-1}W_{p}=0$, therefore $\lim_{n\rightarrow
\infty}x_{4n-3}=0$. Similarly, one can  prove that the limits of
the other subsequences is zero. Hence $(x_n)_{n=-3}^\infty$
converges to zero.\end{proof}
\begin{thm}\label{Th_5}
If $A=-\alpha$, then every solution of equation \eqref{e_1} is
unbounded.
\end{thm}
\begin{proof}
If we replace $\alpha$ by $-A$ in the first term of equation
\eqref{e_2}, we obtain
\begin{align*}
x_{4n-3}&=\frac{d(-A)^n\prod_{p=0}^{n-2}\Big(A^{2p+2}+B
bd\sum_{i=0}^{2p+1}A^i(-A)^{2p+1-i}\Big)}{\prod_{p=0}^{n-1}\Big(A^{2p+1}+B
bd\sum_{i=0}^{2p}A^i(-A)^{2p-i}\Big)}\\
&=\frac{d(-A)^n\prod_{p=0}^{n-2}A^{2p+1}\Big(A+B
bd\sum_{i=0}^{2p+1}(-1)^{2p+1-i}\Big)}{\prod_{p=0}^{n-1}A^{2p}\Big(A+B
bd\sum_{i=0}^{2p}(-1)^{2p-i}\Big)}\\
&=\frac{d(-A)^n\prod_{p=0}^{n-2}A^{2p+1}\Big(A+B
bd\sum_{k=0}^{2p+1}(-1)^{k}\Big)}{\prod_{p=0}^{n-1}A^{2p}\Big(A+B
bd\sum_{k=0}^{2p}(-1)^{k}\Big)}\\
&=
\frac{d(-A)^n\prod_{p=0}^{n-2}A^{2p+2}}{\prod_{p=0}^{n-1}A^{2p}\Big(A+B
bd\Big)}\\
&=\frac{d(-A)^n A^{2n-2}\prod_{p=0}^{n-2}A^{2p}}{
A^{2n-2}\prod_{p=0}^{n-2}A^{2p}\Big(A+B bd\Big)^{n}}\\
&=\frac{ d(-A)^n}{\Big(A+B
bd\Big)^{n}}\\
&=\frac{ d}{\Big(-1-\frac{B
bd}{A}\Big)^{n}}\\
&=\frac{ d}{(-1-e^{-1})^{n}}
\end{align*}
 Hence:

(i) If $|1+e^{-1}|>1$, then $(x_{4n-3})_n$ converges to zero.

(ii) If $|1+e^{-1}|<1$, then $(x_{4n-3})_n$ is not bounded and
$\lim_{n\rightarrow\infty}|x_{4n-3}|=+\infty$.

In other hand, If we replace $\alpha$ by $-A$ in the first term of
equation \eqref{e_3}, we obtain
\begin{align*}
x_{4n-1}&=\frac{b(-A)^n\prod_{p=0}^{n-1}\Big(A^{2p+1}+B
bd\sum_{i=0}^{2p}A^i(-A)^{2p-i}\Big)}{\prod_{p=0}^{n-1}\Big(A^{2p+2}+B
bd\sum_{i=0}^{2p+1}A^i(-A)^{2p+1-i}\Big)}\\
&= \frac{b(-A)^n\prod_{p=0}^{n-1}A^{2p}\Big(A+B
bd\sum_{i=0}^{2p}(-1)^{2p-i}\Big)}{\prod_{p=0}^{n-1}A^{2p+1}\Big(A+B
bd\sum_{i=0}^{2p+1}(-1)^{2p+1-i}\Big)}\\
&= b(-A)^n\prod_{p=0}^{n-1}\Big(\frac{A^{2p}(A+B
bd)}{A^{2p+2}}\Big)\\
&= b(-A)^n\prod_{p=0}^{n-1}\Big(\frac{A^{2p}(A+B
bd)}{A^{2p+2}}\Big)\\
 &=b(-1)^n\Big(1+\frac{B
bd}{A}\Big)^n\\
 &= b(-1-e^{-1})^n
\end{align*}
 Hence:

(iii) If $|1+e^{-1}|>1$, then $(x_{4n-1})_n$ is not bounded and
$\lim_{n\rightarrow\infty}|x_{4n-3}|=+\infty$.

(iv) If $|1+e^{-1}|<1$, then $(x_{4n-1})_n$ converges to zero.

Combining propositions (i), (ii), (iii) and (iv), the proof is
ended.
\end{proof}
\begin{rem}\label{Rem_2}\begin{enumerate}
\item Similarly, we prove that we have either $(|x_{4n-2}|)_n$
diverges to $\infty$ and $(x_{4n})_n$ converges to zero or
$(x_{4n-2})_n$ converges to zero and $(|x_{4n}|)_n$ diverges to
$\infty$. \item Note that in the case $A=-\alpha$, the initial
conditions intervene in the nature of the subsequences of
$(x_n)_{n=-3}^\infty$, therefore in the nature of the solution
$(x_n)_{n=-3}^\infty$ itself.
\end{enumerate}
\end{rem}
\subsection{The case $|\frac{A}{\alpha}|<1$}
\begin{thm}\label{Th_6}
If $|\frac{A}{\alpha}|<1$, then for every solution
$(x_n)_{n=-3}^\infty$ of the equation \eqref{e_1}, the
subsequences $(x_{4n-3})_n$, $(x_{4n-1})_n$, $(x_{4n-2})_n$ and
$(x_{4n})_n$ converge.
\end{thm}
\begin{proof}
Let us prove that $(x_{4n-3})_n$ converge.
\begin{align*}
 x_{4n-3}&=\frac{d\alpha^n(A-\alpha)\prod_{p=0}^{n-2}\Big(A^{2p+2}(A-\alpha+B bd)-B bd\alpha^{2p+2}\Big)}
 {\prod_{p=0}^{n-1}\Big(A^{2p+1}(A-\alpha+B bd)-B
 bd\alpha^{2p+1}\Big)}\\
 &=\frac{d\alpha^n(A-\alpha)\prod_{p=0}^{n-2}-B
bd\alpha^{2p+2}\Big(-\frac{A-\alpha+B
bd}{Bbd}(\frac{A}{\alpha})^{2p+2}+1\Big)}{\prod_{p=0}^{n-1}-B
bd\alpha^{2p+1}\Big(-\frac{A-\alpha+B
bd}{Bbd}(\frac{A}{\alpha})^{2p+1}+1\Big)}\\
&=\frac{d\alpha^n(A-\alpha)\prod_{p=0}^{n-2}-B
bd\alpha^{2p+2}\prod_{p=0}^{n-2}\Big(1-\frac{A-\alpha+B
bd}{Bbd}(\frac{A}{\alpha})^{2p+2}\Big)}{\prod_{p=0}^{n-1}-B
bd\alpha^{2p+1}\prod_{p=0}^{n-1}\Big(1-\frac{A-\alpha+B
bd}{Bbd}(\frac{A}{\alpha})^{2p+1}\Big)}\\
&=\frac{d\alpha^n(A-\alpha)\prod_{p=0}^{n-2}\alpha^{2p+1}\alpha^{n-1}\prod_{p=0}^{n-2}\Big(1-\frac{A-\alpha+B
bd}{Bbd}(\frac{A}{\alpha})^{2p+2}\Big)}{-B
bd\prod_{p=0}^{n-2}\alpha^{2p+1}\alpha^{2n-1}\prod_{p=0}^{n-1}\Big(1-\frac{A-\alpha+B
bd}{Bbd}(\frac{A}{\alpha})^{2p+1}\Big)}\\
&=\frac{\alpha-A}{B b\Big(1-\frac{A-\alpha+B
bd}{Bbd}(\frac{A}{\alpha})^{2n-1}\Big)}\prod_{p=0}^{n-2}\Big(\frac{1-\frac{A-\alpha+B
bd}{Bbd}(\frac{A}{\alpha})^{2p+2}}{1-\frac{A-\alpha+B
bd}{Bbd}(\frac{A}{\alpha})^{2p+1}}\Big)
\end{align*}
 Denote $\gamma=\frac{A-\alpha+B bd}{Bbd}$,  $\lambda=\frac{A}{\alpha}$ and
  $V_p=\frac{1-\gamma\lambda^{2p+2}}{1-\gamma\lambda^{2p+1}}$.
Using these notations, we obtain
$$x_{4n-3}=\frac{\alpha-A}{Bb(1-\gamma\lambda^{2n-1})}\prod_{p=0}^{n-2}V_p$$
Since $|\lambda|<1$, then the sequence
$\Big(\frac{\alpha-A}{Bb(1-\gamma\lambda^{2n-1})}\Big)_n$
converges toward $\frac{\alpha-A}{Bb}$. By the Taylor expansion,
we obtain
$$V_p=\frac{1-\gamma\lambda^{2p+2}}{1-\gamma\lambda^{2p+1}}=
(1-\gamma\lambda^{2p+2})(1+\gamma\lambda^{2p+1}+o(\lambda^{2p+1}))=1+\gamma(1-\lambda)\lambda^{2p+1}+o(\lambda^{2p+1}).$$
Then $\ln(V_p)\sim \gamma(1-\lambda)\lambda^{2p+1}$, which is the
general term of a convergent series, then the sequence
$(x_{4n-3})_n$ is convergent. Similarly, one can prove that the
other subsequences converge.
\end{proof}
\begin{rem}\noindent Sequences $(x_{4n-3})_n$ and $(x_{4n-1})_n$ are related by
the equations:
\begin{equation}\label{e_4}x_{4(n+1)-3}=\frac{\alpha x_{4n-3}}{A+B x_{4n-1}x_{4n-3}},\end{equation}
and
\begin{equation}\label{e_5}x_{4(n+1)-1}=\frac{\alpha x_{4n-1}}{A+B x_{4(n+1)-3}x_{4n-1}}.\end{equation}
Denote by $l_{3}$, $l_{2}$, $l_{1}$ and $l_{0}$ respectively, the
limits of the subsequences   $(x_{4n-3})_n$, $(x_{4n-2})_n$,
$(x_{4n-1})_n$ and $(x_{4n})_n$ .

Passing to the limit as $n$ goes to infinity in the equation
\eqref{e_4}, we obtain:

$l_{3}=\frac{\alpha l_{3}}{A+Bl_{3}l_{1}}$, then
$$(S_1)\,:\begin{cases}
l_{3}=0,\\
or\\
l_{1}=\frac{\alpha-A}{Bl_{3}}.
\end{cases}$$
Passing to the limit as $n$ goes to infinity in the equation
\eqref{e_5}, we obtain:

$l_{1}=\frac{\alpha l_{1}}{A+Bl_{3}l_{1}}$, then
$$(S_2)\,:\begin{cases}
l_{1}=0,\\
or\\
l_{3}=\frac{\alpha-A}{Bl_{1}}.
\end{cases}$$
Combining systems $(S_1)$ and $(S_2)$, since $\alpha \neq A$, we
obtain:

either $$l_{3}=l_{1}=0$$ \qquad \qquad \qquad or\quad
$$(S):\,\begin{cases}
l_{3}=\frac{\alpha-A}{Bl_{1}},\\
and\\
l_{1}=\frac{\alpha-A}{Bl_{3}}.
\end{cases}$$
The proposition $l_{3}=l_{1}=0$ contradicts the fact that the
infinite product $\prod_{p\geq 0}V_p$ converges, in fact if
$\lim_{n\rightarrow \infty}\prod_{p=0}^{n}V_p=0$, then
$\lim_{n\rightarrow \infty}\sum_{p=0}^{n}\ln(V_p)=-\infty$ and
this contradicts the fact that the series $\sum_{p\geq 0}\ln(V_p)$
converges. Hence the only possibility is that $l_{1}\neq 0$,
$l_{3}\neq 0$ and $(S):\,\begin{cases}
l_{3}=\frac{\alpha-A}{Bl_{1}},\\
and\\
l_{1}=\frac{\alpha-A}{Bl_{3}}.
\end{cases}$

In fact $l_{3}=\frac{\alpha-A}{Bl_{1}}$ is equivalent to
$l_{1}=\frac{\alpha-A}{Bl_{3}}$, then $(S)$ is equivalent to
$l_{3}=\frac{\alpha-A}{Bl_{1}}$. Consider the function $f$ defined
on $\mathbb{R}^*$ as $f(x)=\frac{\alpha-A}{Bx}$, we have $fof=Id$
and, $l_{1}$ and $l_{3}$ are related by $f(l_{1})=l_{3}$.
Similarly, we prove that $l_{0}$ and $l_{2}$ are related by the
relation $f(l_{0})=l_{2}$.
 \end{rem}
\section{Periodicity}
\begin{defn}\label{defp}
A solution $(x_n)_{n=-3}^\infty$ of \eqref{e_1} is called periodic
with period $p$ if there exists an integer $p$, such that
\begin{equation}\label{e_p} x_{n+p} = x_n,\:\:\forall n\geq -3 \end{equation}
 A solution is called periodic with prime period $p$ if $p$ is
the smallest positive integer for which \eqref{e_p}
holds.\end{defn}
 In the sequel, we need
the following lemma, \cite{gla}, which describes when a solution
of $\eqref {e_1}$ converges to a periodic solution of $\eqref
{e_1}$.
\begin{lem}\label{lemp}
Let $(x_n)_{n=-3}^\infty$ be a solution of \eqref{e_1}. Suppose
there exist real numbers $l_3,l_2,l_1,l_0$ such that
$$\lim_{n\longrightarrow +\infty}x_{4n+j}=l_j\:\: \mbox{for
all}\:\:j=-3,\ldots ,0$$ Let $(y_n)_{n=-3}^\infty$ be the
period-$4$ sequence of real numbers such that
$$y_{4n+j}=l_j\:\: \mbox{for
all}\:\:j\geq-3$$ Then the following statements are true:
\begin{enumerate}
\item $(y_n)_{n=-3}^\infty$ is a period-4 solution of \eqref{e_1}.
\item $\lim_{n\longrightarrow +\infty}x_{4n+j}=y_j \mbox{ for all
}\:\:j\geq-3$ \end{enumerate}
\end{lem}
Now, the field is ready to state the following theorem:
\begin{thm}\noindent
\begin{enumerate}
\item \label{it1}If $|\frac{A}{\alpha}|>1$, $A-\alpha+Bbd\neq 0$
and $A-\alpha+Bac\neq 0$, then the equation \eqref{e_1} has no
periodic-p solution, for all $p\geq 2$.
 \item If $|\frac{A}{\alpha}|>1$, $A-\alpha+Bbd= 0$ and
$A-\alpha+Bac=0$, then the equation \eqref{e_1} has a periodic-4
solution.
 \item If $|\frac{A}{\alpha}|= 1$ then the equation \eqref{e_1} has no
periodic-p solution, for all $p\geq 2$.
 \item If $|\frac{A}{\alpha}|<1$ then the equation \eqref{e_1} has a
periodic-4 solution.
\end{enumerate}
\end{thm}
\begin{proof}\noindent
\begin{enumerate}
\item If $|\frac{A}{\alpha}|>1$, $A-\alpha+Bbd\neq 0$ and
$A-\alpha+Bac\neq 0$, then, by theorem \eqref{Th_3}, every
solution of \eqref{e_1} converges to 0, hence, the solution is not
allowed to be periodic. \item If $|\frac{A}{\alpha}|>1$,
$A-\alpha+Bbd= 0$ and $A-\alpha+Bac=0$, then, by theorem
\eqref{Th_3},$\lim_{n\longrightarrow +\infty}x_{4n-3}=d$,
$\lim_{n\longrightarrow +\infty}x_{4n-2}=c$,
$\lim_{n\longrightarrow +\infty}x_{4n-1}=b$,
$\lim_{n\longrightarrow +\infty}x_{4n}=a$. Applying the lemma
\eqref{lemp}, the sequence: $d,c,b,a,d,c,b,a\ldots $ is a
periodic-4 solution.
 \item The case $A=\alpha$ is similar to \eqref{it1}. Suppose that  $A=-\alpha$, by contradiction,  assume that the
 equation \eqref{e_1} has a periodic-p solution
 $(x_n)_{n=-3}^\infty$. Let $x_{n+1},\ldots , x_{n+p}$ be a $p$
 consecutive values of this solutions and let $M=\max\{x_{n+1},\ldots ,
 x_{n+p}\}$, it follows that for all $k\geq-3$, $x_k\leq M$.
 Contradiction with theorem \eqref{Th_5}.
 \item If $|\frac{A}{\alpha}|<1$ then, by theorem \eqref{Th_6}, there exist real numbers $l_3,l_2,l_1,l_0$ such that
$$\lim_{n\longrightarrow +\infty}x_{4n+j}=l_j\:\: \mbox{for
all}\:\:j=-3,\ldots ,0$$ Applying lemma \eqref{lemp}, the sequence
$l_3,l_2,l_1,l_0,l_3,l_2,l_1,l_0\ldots$ is a periodic-4 solution
of \eqref{e_1}.
\end{enumerate}
\end{proof}
\begin{rem}\label{remp}From the previous proof, we
deduce:\begin{enumerate}\item If $|\frac{A}{\alpha}|>1$,
$A-\alpha+Bbd= 0$ and $A-\alpha+Bac=0$, $d=b$, $c=a$ and $b\neq c$
then \eqref{e_1} has a 2 prime periodic solution $a,b,a,b\ldots$.
 \item If $|\frac{A}{\alpha}|<1$, $A-\alpha+Bbd = 0$ and
$A-\alpha+Bac = 0$ then, from the proof of theorem \eqref{Th_6},
we deduce the values of the real numbers $l_3,l_2,l_1,l_0$, thus:
$l_3=\frac{\alpha-A}{Bb}$, $l_2=\frac{\alpha-A}{Ba}$,
$l_1=\frac{\alpha-A}{Bd}$, $l_0=\frac{\alpha-A}{Bc}$.
\end{enumerate}\end{rem}
\section{Examples}
\begin{ex}\label{exp1}
In this example, we illustrate the case $|\frac{A}{\alpha}|>1$,
$A-\alpha+Bbd\neq 0$ and $A-\alpha+Bac\neq 0$, we choose $a=3$;
$b=-4$; $c=2$ ; $d=-1$ ; $B=1$ ; $A=1.05$; $\alpha =1$.\\
\begin{figure}\label{fig1}
\begin{center}
      \includegraphics{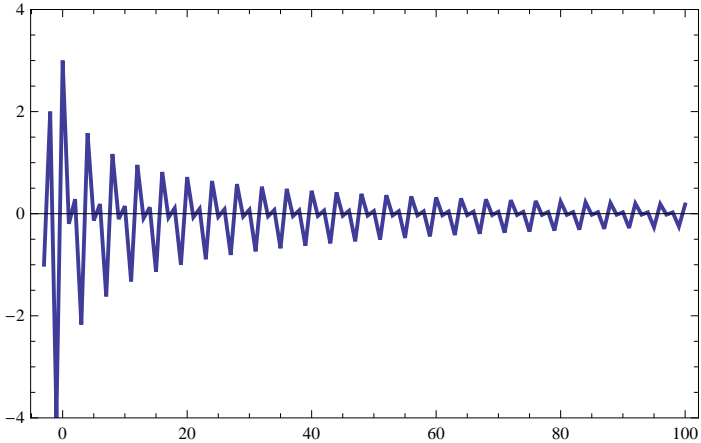}
   \end{center}
   \caption{$|\frac{A}{\alpha}|>1$, the solution is converging to zero}
\end{figure}
We remark in the figure \eqref{fig1} that the solution is
oscillating about zero with a decreasing amplitude. In fact, the
solution has to converge to zero, According to theorem
\eqref{Th_3}.
\end{ex}

\begin{ex}\label{exp2}
In this example, we illustrate the case $|\frac{A}{\alpha}|>1$,
$A-\alpha+Bbd= 0$ and $A-\alpha+Bac = 0$, we choose $a=c=2$;
$b=d=-2$ ; $B=-2$ ; $A=9$; $\alpha =1$.\\
\begin{figure}\label{fig2}
\begin{center}
      \includegraphics{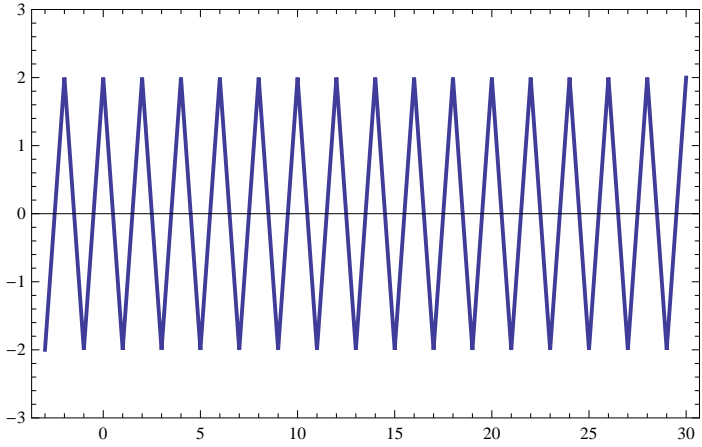}
   \end{center}
   \caption{$|\frac{A}{\alpha}|>1$, 2 prime periodic solution}
\end{figure}
We see, figure \eqref{fig2}, that the obtained solution is a 2
prime periodic solution. This is coherent with remark
\eqref{remp}.
\end{ex}
\begin{ex}\label{exp3}
In this example, we illustrate the case $A=-\alpha$, we choose $a
= 0.1$ ; $b = 0.2$; $c = 0.3$ ; $d = -0.4$ ; $B = 1$ ; $A = 0.5$;
$\alpha = -0.5$\\
\begin{figure}\label{fig3}
\begin{center}
      \includegraphics{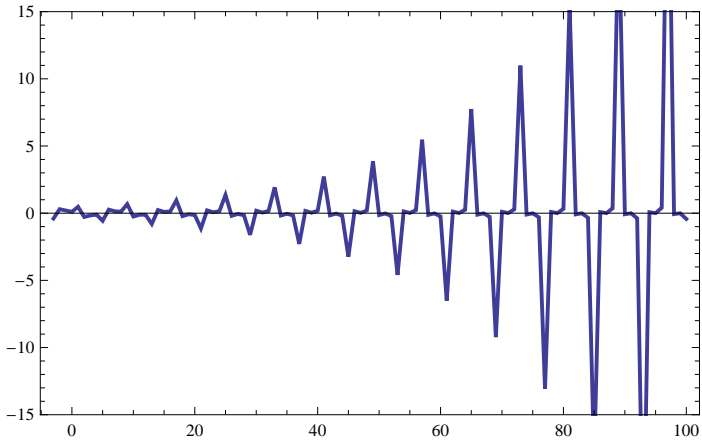}
   \end{center}
   \caption{$A=-\alpha $, the solution is unbounded}
\end{figure}
We remark, figure \eqref{fig3}, that the solution is oscillating
about zero with an increasing amplitude. By theorem \eqref{Th_5},
it is an unbounded solution.
\end{ex}

\begin{ex}\label{exp4}
In this example, we illustrate the case $|\frac{A}{\alpha}|<1$,
$A-\alpha+Bbd\neq 0$ and $A-\alpha+Bac\neq 0$, we choose  $a =
-1.2 ; b = 0.4; c = -0.3 ; d = 0.9 ; B = 1 ; A = 0.64; \alpha =
1$\\
\begin{figure}\label{fig4}
\begin{center}
      \includegraphics{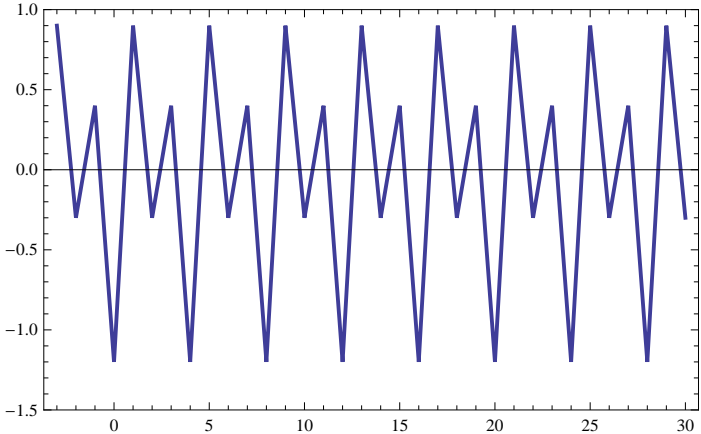}
   \end{center}
   \caption{$|\frac{A}{\alpha}|<1$, 4 prime periodic solution}
\end{figure}
We obtain a 4 prime periodic solution, Figure \eqref{fig4}.
According to the remark \eqref{remp}, the solution is the 4 prime
periodic sequence
$$0.9, 0.4, -0.3, -1.2, 0.9, \ldots$$
\end{ex}

\end{document}